\documentclass[11pt,letterpaper]{amsart}

\usepackage[margin=1.25in]{geometry}                

\usepackage{mathabx}
\usepackage{amsmath,amssymb,amsthm,amscd}
\usepackage{colonequals}

\usepackage{graphicx}
\usepackage{tikz}
\usepackage{tikz-cd}
\usetikzlibrary{decorations.markings}
\usetikzlibrary{plotmarks}
\usepackage{epstopdf}
\usepackage{import} 
\usepackage{transparent}

\usepackage{multicol}

\usepackage{url}
\usepackage{color}
\usepackage[colorlinks,citecolor=blue,linkcolor=red!80!black]{hyperref}
\usepackage[nameinlink]{cleveref}
\usepackage{comment}

\usepackage{marginnote}
\usepackage{todonotes}

\theoremstyle{plain}
\newtheorem{theorem}{Theorem}[section]
\newtheorem{lemma}[theorem]{Lemma}

\newtheorem{proposition}[theorem]{Proposition}

\crefname{claim}{Claim}{Claims}
\newtheorem*{claim*}{Claim}

\newcounter{theoremalph}

\newtheorem{thmAlph}[theoremalph]{Theorem}

\theoremstyle{definition}

\newtheorem{question}[theorem]{Question}

\newtheorem{example}[theorem]{Example}

\makeatletter
\let\c@equation\c@theorem
\makeatother
\numberwithin{equation}{section}

\newcommand{\la}{\langle}
\newcommand{\ra}{\rangle}

\newcommand{\N}{\mathbb{N}}
\newcommand{\Z}{\mathbb{Z}}

\newcommand{\R}{\mathbb{R}}

\newcommand{\free}{\mathbb{F}}
\newcommand{\out}{\mathrm{Out}(\free)}

\newcommand{\FreeF}{\mathcal{F}}
\newcommand{\relFreeS}[1]{\mathcal{FS}(\free, #1)}

\newcommand{\cv}{\mathcal{O}}

\newcommand{\relcv}[1]{\mathcal{O}(\free, #1)}
\newcommand{\relCV}[1]{\mathbb{P}\mathcal{O}(\free, #1)}

\newcommand{\core}[1]{C_w(T_{#1})}
\newcommand{\diam}{\mathrm{diam}}
\DeclareMathOperator{\BBT}{BBT} 
\DeclareMathOperator{\cii}{\text{combinatorial isoperimetric inequality}}

\title{Combinatorial isoperimetric inequality for the free factor complex}
\author[Gupta]{Radhika Gupta}
\address{School of Mathematics, Tata Institute of Fundamental Research, Mumbai}
\email{rgupta@math.tifr.res.in}
\date{}                                           

\begin{document}
\maketitle

\begin{abstract}
 We show that the free factor complex of the free group of rank $n \geq 4$ does not satisfy a $\cii$: that is, for every $N \in \N$, there is a loop $c_N$ of length 4 in the free factor complex such that the number of 2-simplices required to fill $c_N$ grows at least as a linear function of $N$. To prove the result, we construct a coarsely Lipschitz map from the `upward link' of a free factor to $\mathbb{Z}$. 
\end{abstract}

\section{Introduction}

Webb \cite{Webb} showed that arc complexes associated to almost all hyperbolic surfaces do not admit a \emph{$\mathrm{CAT}(0)$ metric with finitely many shapes}. That is, they do not admit CAT(0) metrics which have finitely many isometry types of simplices, in the induced metric. He proved this by showing that arc complexes do not satisfy a \emph{combinatorial isoperimetric inequality}. In contrast, he showed that the curve complex of a hyperbolic surface satisfies a \emph{linear} $\cii$.

Analogously, the group of outer automorphisms of a free group $\out$ acts on some hyperbolic simplicial complexes, like the free splitting complex, cyclic splitting complex and the free factor complex. Webb showed that the free splitting complex and the cyclic splitting complex also do not satisfy a $\cii$. In this article, we show that the \emph{free factor complex} follows suit with the other two $\out$-complexes. Thus still leaving open the question: is there a cocompact complex for $\out$, analogous to the curve complex, that satisfies a linear combinatorial isoperimetric inequality?

The free factor complex $\FreeF$ associated to $\free$ is the simplicial complex whose vertices are conjugacy classes of proper free factors of $\free$ and a collection $\{[A_1], \ldots, [A_k]\}$ of vertices spans a simplex if $A_1 \subset A_2 \subset \cdots \subset A_k$, where the inclusions are up to conjugation. When it is clear from context, we drop $[.]$ to denote the conjugacy class. We show: 

\begin{thmAlph}
Let $\FreeF$ be the free factor complex of free group of rank $n \geq 4$. There exists a constant $C = 2n-4$ and a family of loops $c_N$, for $N \in \N$, of combinatorial length 4 in $\FreeF^{(1)}$ such that the following holds: whenever $P$ is a triangulation of a disc and $f\colon P \to \FreeF^{(2)}$ is a simplicial map with $f|_{\partial P}$ mapping bijectively onto $c_N$, then $P$ must have at least $(N-2C)/3$ triangles. \end{thmAlph}

The free factor complex for the free group of rank 3 is 1-dimensional, so we slightly modify the definition of the free factor complex, up to quasi-isometry, to prove a statement as above. This is done in \Cref{n=3}. 

All of the simplicial complexes mentioned above are Gromov hyperbolic spaces by results of \cite{MM:CurveComplex, HM:FreeSplittingComplex, BF:Hyperbolicity, M:CyclicS} and a Gromov hyperbolic space is characterised by having a linear \emph{coarse} isoperimetric inequality. That coarse and combinatorial isoperimetric inequalities are different can be seen in the example shown in \Cref{CombVsCoarse}, where we consider all 2-simplices to be isometric. Indeed, the space does not satisfy a combinatorial isoperimetric inequality, that is, there are loops of combinatorial length four that need arbitrarily many triangles to homotope to a point but the space has a constant coarse combinatorial inequality. 

\begin{figure}[ht]
    \centering{
    \def\svgwidth{.5\columnwidth}
\begingroup%
  \makeatletter%
  \providecommand\color[2][]{%
    \errmessage{(Inkscape) Color is used for the text in Inkscape, but the package 'color.sty' is not loaded}%
    \renewcommand\color[2][]{}%
  }%
  \providecommand\transparent[1]{%
    \errmessage{(Inkscape) Transparency is used (non-zero) for the text in Inkscape, but the package 'transparent.sty' is not loaded}%
    \renewcommand\transparent[1]{}%
  }%
  \providecommand\rotatebox[2]{#2}%
  \newcommand*\fsize{\dimexpr\f@size pt\relax}%
  \newcommand*\lineheight[1]{\fontsize{\fsize}{#1\fsize}\selectfont}%
  \ifx\svgwidth\undefined%
    \setlength{\unitlength}{191.22368882bp}%
    \ifx\svgscale\undefined%
      \relax%
    \else%
      \setlength{\unitlength}{\unitlength * \real{\svgscale}}%
    \fi%
  \else%
    \setlength{\unitlength}{\svgwidth}%
  \fi%
  \global\let\svgwidth\undefined%
  \global\let\svgscale\undefined%
  \makeatother%
  \begin{picture}(1,0.37172388)%
    \lineheight{1}%
    \setlength\tabcolsep{0pt}%
    \put(0,0){\includegraphics[width=\unitlength,page=1]{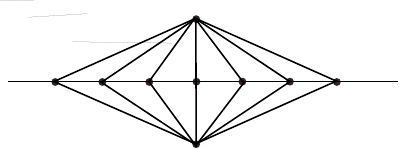}}%
    \put(0.13588112,0.25097837){\color[rgb]{0,0,0}\makebox(0,0)[lt]{\lineheight{1.25}\smash{\begin{tabular}[t]{l}$\ldots$\end{tabular}}}}%
    \put(0.76203655,0.24857893){\color[rgb]{0,0,0}\makebox(0,0)[lt]{\lineheight{1.25}\smash{\begin{tabular}[t]{l}$\ldots$\end{tabular}}}}%
    \put(0.76727934,0.05933814){\color[rgb]{0,0,0}\makebox(0,0)[lt]{\lineheight{1.25}\smash{\begin{tabular}[t]{l}$\ldots$\end{tabular}}}}%
    \put(0.13466183,0.06197464){\color[rgb]{0,0,0}\makebox(0,0)[lt]{\lineheight{1.25}\smash{\begin{tabular}[t]{l}$\ldots$\end{tabular}}}}%
  \end{picture}%
\endgroup%

    \caption{Combinatorial vs coarse isoperimetric inequality.}
    \label{CombVsCoarse}
    }
\end{figure}
 
Note that the arc complex of a hyperbolic surface and free splitting complex of a free group are contractible spaces \cite{Harer:Stability, McCullough,Hatcher:ArcComplex,Hatcher:FSContractible}. Therefore, these spaces may still admit a CAT(0) metric, though not one with finitely many shapes. We refer the reader to \cite{Webb} for motivation to study CAT(0) metrics on these spaces.  On the contrary, the curve complex and the free factor complex are homotopy equivalent to a wedge of spheres \cite{H:VCD, Ivanov, BruckGupta} so neither space can admit a CAT(0) metric.

The proof strategy for the theorem is as follows. For any $N >0$, we construct an explicit loop $c_N$ of length 4 with vertices $A_0, A, A_N, B$, where $A$ is a rank one free factor. See \Cref{loop}. We show that any path between $A_0$ and $A_N$ in the link of $A$ has length bounded from below by a linear function of $N$. We do this by defining (see \Cref{coarsely Lip}) a coarsely Lipschitz map from the link of $A$ to $\Z$. Using this, we conclude that the number of triangles needed to cap off any disc with boundary $c_N$ is at least a linear function of $N$.   

In fact, in \Cref{sec:CL function} we define a family of coarsely Lipschitz maps from the upward link, denoted $\FreeF^{\uparrow}(A)$, of a free factor $A$ of corank at least 3  to $\Z$. 
We choose a complementary free factor $B$, a cyclically reduced filling element $w$ in $B$ that is not a power, and a primitive element $b$ in $B$. Then we define $\Psi_{w,b,B} \colon \FreeF^{\uparrow}(A) \to \Z$ to roughly measure how much $b$ gets conjugated by $w$, in other words how much $b$ twists about $w$, in a basis of $\free$ that contains a basis of $X \in \FreeF^{\uparrow}(A)$ as a subbasis. 

\begin{figure}[ht]
    \centering{
    \def\svgwidth{.5\columnwidth}
\begingroup%
  \makeatletter%
  \providecommand\color[2][]{%
    \errmessage{(Inkscape) Color is used for the text in Inkscape, but the package 'color.sty' is not loaded}%
    \renewcommand\color[2][]{}%
  }%
  \providecommand\transparent[1]{%
    \errmessage{(Inkscape) Transparency is used (non-zero) for the text in Inkscape, but the package 'transparent.sty' is not loaded}%
    \renewcommand\transparent[1]{}%
  }%
  \providecommand\rotatebox[2]{#2}%
  \newcommand*\fsize{\dimexpr\f@size pt\relax}%
  \newcommand*\lineheight[1]{\fontsize{\fsize}{#1\fsize}\selectfont}%
  \ifx\svgwidth\undefined%
    \setlength{\unitlength}{275.10684526bp}%
    \ifx\svgscale\undefined%
      \relax%
    \else%
      \setlength{\unitlength}{\unitlength * \real{\svgscale}}%
    \fi%
  \else%
    \setlength{\unitlength}{\svgwidth}%
  \fi%
  \global\let\svgwidth\undefined%
  \global\let\svgscale\undefined%
  \makeatother%
  \begin{picture}(1,0.61382785)%
    \lineheight{1}%
    \setlength\tabcolsep{0pt}%
    \put(0,0){\includegraphics[width=\unitlength,page=1]{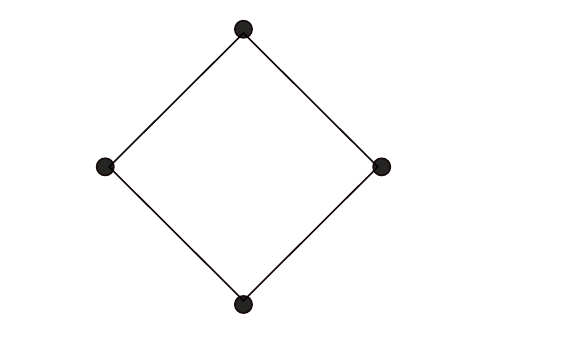}}%
    \put(0.39594297,0.00426471){\color[rgb]{0,0,0}\makebox(0,0)[lt]{\lineheight{1.25}\smash{\begin{tabular}[t]{l}$[\la a \ra]$\end{tabular}}}}%
    \put(0.69824226,0.31392618){\color[rgb]{0,0,0}\makebox(0,0)[lt]{\lineheight{1.25}\smash{\begin{tabular}[t]{l}$[\la a, w^N b w^{-N}\ra]$\end{tabular}}}}%
    \put(-0.00181038,0.31587343){\color[rgb]{0,0,0}\makebox(0,0)[lt]{\lineheight{1.25}\smash{\begin{tabular}[t]{l}$[\la a, b \ra]$\end{tabular}}}}%
    \put(0.38764874,0.59725758){\color[rgb]{0,0,0}\makebox(0,0)[lt]{\lineheight{1.25}\smash{\begin{tabular}[t]{l}$[\la b \ra]$\end{tabular}}}}%
  \end{picture}%
\endgroup%

    \caption{The loop $c_N$.}
    \label{loop}
    }
\end{figure}

The free factor complex $\FreeF$ is quasi-isometric to another complex called the \emph{complex of free factor systems}, denoted $\FreeF \FreeF$. The complex $\FreeF \FreeF$ is closely related to the simplicial closure of unreduced Outer space (see \cite{BruckGupta}). A free factor system of $\free$ is a finite collection of the form $\mathcal{A}=\{[A_1],...,[A_k]\},$ where $k >0$, each $A_i$ is a proper, non-trivial free factor of $\free$, such that there exists a free factorization $\free = A_1 \ast \cdots A_k \ast F_M$.  There is a partial ordering on the set of free factor systems given as follows: $\mathcal{A} \sqsubseteq \mathcal{A}'$ if for every $[A_i] \in \mathcal{A}$ there exists $[A'_j] \in \mathcal{A'}$ such that $A_i \subseteq A_j'$ up to conjugation. The vertices of $\FreeF \FreeF$ are free factor systems and two vertices $\mathcal{A}_1$ and $\mathcal{A}_2$ are joined by an edge if $\mathcal{A}_1 \sqsubseteq \mathcal{A}_2$ or vice versa. 

Now recall the length 4 loop $c_N$ in $\FreeF$. It can also be viewed as a loop in $\FreeF \FreeF$. However, in this complex, the vertex $\{[\la a \ra], [\la b \ra]\}$ is connected to both $A_0$ and $A_N$ and hence the loops $c_N$ can be capped off with 4 triangles. Thus it is natural to ask: 
\begin{question}
	Does the complex of free factor systems  $\FreeF \FreeF$ satisfy some kind of combinatorial isoperimetric bound?
\end{question}

\subsubsection*{Acknowledgments} The author would like to thank Richard Webb for asking her this question and useful discussions, Mladen Bestvina for sharing his ideas and Jean Pierrre Mutanguha for reading a draft. She would also like to thank the referee for suggestions that shortened the proof of the main proposition. The author was supported by  the Department of Atomic Energy, Government of India, under project no.12-R\&D-TIFR-5.01-0500, by an endowment of the Infosys Foundation and SERB research grant SRG/2023/000123. 

\section{Background}

\subsection{Combinatorial isoperimetric inequalities}\label{subsec:CII}
We recall some definitions from \cite{Webb}. 

Let $P$ and $K$ be simplicial complexes. A \emph{simplicial map} is a map $c \colon P \to K$ that sends each simplex of $P$ to a simplex of $K$ by a linear map taking vertices to vertices. Let $K^{(i)}$ denote the $i$-skeleton of $K$. A \emph{combinatorial loop} $c$ in $K$ is a sequence of vertices $(v_1, \ldots, v_k)$ in $K$ where $v_k$ is adjacent (or equal) to $v_1$ and $v_i$ is adjacent (or equal) to $v_{i+1}$ for $1 \leq i \leq k-1$. The \emph{combinatorial length} $l_C(c)$ of $c$ is equal to $k$. A combinatorial loop $c$ of combinatorial length $k$ can also be described as a simplicial map $c \colon P \to K$, where $P$ is a triangulation of $S^1$ with $k$ 1-simplices. 

Let $D^2$ be the closed unit disc with boundary $S^1$. We say that a combinatorial loop $c$ can be \emph{capped off with at most $m$ triangles} if there is a triangulation $P$ of $D^2$ into at most $m$ 2-simplices and a simplicial map $c' \colon P \to K$ such that $c'|_{S^1} = c$. 

A function $f \colon \N \to \N$ is called a \emph{combinatorial isoperimetric bound} for $K$ if every combinatorial loop $c$ in $K$ can be capped off with at most $f(l_C(c))$ many triangles. We say $K$ satisfies a \emph{linear $\cii$} if there exists a combinatorial isoperimetric bound $f$ for $K$ such that $f(n) = \mathrm{O(n)}$. We say \emph{$K$ satisfies no combinatorial isoperimetric inequality} if no combinatorial isoperimetric bound for $K$ exists. 

\subsection{Outer space}
We say an $\R$-tree is an $\free$-tree if it admits a non-trivial action of $\free$. We denote by $\cv$ the (unprojectivized) Outer space of $\free$, defined in \cite{CV:OuterSpace}, consisting of minimal, metric, simplicial $\free$-trees up to equivariant isometry. Let $A$ be a free factor of $\free$. Let $\relcv{A}$ be the unprojectivised \emph{Outer space relative to $A$} (\cite{GL:relCV}), that is, the space of minimal, metric, simplicial $\free$-trees with trivial edge stabilizers and vertex group system equal to the conjugacy class of $A$, up to equivariant isomorphism. Here the \emph{vertex group system} of a simplicial $\free$-tree is the collection of conjugacy classes of its vertex stabilizers. Let $\relCV{A}$ denote the projectivized outer space relative to $A$. For $T \in \relcv{A}$, the \emph{covolume of $T$} is the sum of the lengths of edges of $T/ \free$, which is a finite graph. Then we think of points in $\relCV{A}$ as covolume 1 trees.

Let $\relFreeS{A}$ be the free splitting graph relative to $A$ (\cite{HM:Relative}). A vertex is given by minimal, simplicial $\free$-tree $T$, without the metric, with trivial edge stabilizers and such that $A$ is elliptic in $T$, up to equivariant isomorphism. Two such trees $T_1, T_2$ are joined by an edge if there is an equivariant collection of edges in $T_1$ that can be collapsed to obtain $T_2$. 
There is a natural function from $\relCV{A}$ to $\relFreeS{A}$, given by forgetting the metric on the tree, such that the image is contained in the vertex set of $\relFreeS{A}$.  

\subsection{Train track structure and morphism}
Let $T$ be a simplicial $\free$-tree. A direction $d$
based at $p \in T$ is a component of $T - \{p\}$. A turn is an unordered pair of directions based at the same point. An \emph{illegal turn structure} on $T$ is an equivariant equivalence relation on the set of directions at each point $p \in T$. The classes of this relation are called \emph{gates}. A turn $(d,d')$ is \emph{legal} if $d$ and $d'$ do not belong to the same gate. If in addition there are at least two gates at every point of $T$, then the illegal turn structure is called a \emph{train track structure}. A path is legal if it only crosses legal turns.

Given two $\free$-trees $T$ and $T'$, an $\free$-equivariant
map $f\colon T \to T'$ is called a \emph{morphism} if every
segment of $T$ can be equivariantly subdivided into finitely many subintervals
such that $f$ is an isometry when restricted to each subinterval.
A morphism between $\free$-trees induces an illegal turn structure on the domain $T$.  A morphism is called
\emph{optimal} if there are at least two gates at each point of
$T$. See \cite{BF:Hyperbolicity, BF:Subfactor} for more details. 
\subsection{Folding}
Let $T, T' \in \relcv{A}$ . 
A \emph{folding path with its natural parametrization}
$(T_t)_{t\in\mathbb{R}^+}$, guided by an optimal morphism
$f\colon T\to T'$ can be defined as follows (see \cite[Section 2]{BF:Hyperbolicity} and \cite[Section 3]{GL:relCV}):  Given
$a,b\in T$ with $f(a)=f(b)$, the \emph{identification time} of $a$ and
$b$ is defined as $\tau(a,b)=\sup_{x\in[a,b]}d_{T'}(f(x),f(a))$.
Define $L:=\frac{1}{2}\BBT(f)$, where $\BBT(f)$ is the bounded backtracking constant for $f$. This is the smallest constant $C$ such that the $f$ image of any path $[p,q]$ is contained in a $C$-neighborhood of $[f(p),f(q)]$.     For each $t\in [0,L]$, one defines an
equivalence relation $\sim_t$ by $a\sim_t b$ if $f(a)=f(b)$ and
$\tau(a,b)<t$.  The tree $T_t$ is then a quotient of $T$ by the
equivalence relation $\sim_t$.  The authors of \cite{GL:relCV} prove that for each
$t\in[0,L]$, $T_t$ is an $\mathbb{R}$-tree. 
Let $\overline{T}, \overline{T'}$ be the covolume 1 representatives of $T$ and $T'$ in $\relCV{A}$. Then a folding path $({S_t})_t$ between them is the projection of the folding path between $T$ and $T'$. In other words, $S_t$ is obtained by rescaling $T_t$ to covolume 1.

 \section{Coarsely Lipschitz map to \texorpdfstring{$\Z$}{Z}} \label{sec:CL function}
 Let $A$ be a free factor of $\free$ and let $B$ be a complementary free factor, that is, $\free = A \ast B$. Choose a filling element $w \in B$, that is, $w$ is not contained in any proper free factor of $B$. Also choose $w$ such that it is cyclically reduced and not a power of another element. Choose $b \in B$ primitive, that is, a basis element.  
Let $\FreeF^{\uparrow}(A)$ be the subcomplex of $\FreeF$ whose vertices are given by conjugacy classes of free factors that properly contain $A$ up to conjugation. In other words, $\FreeF^{\uparrow}(A)$ is the `upward link' of the vertex corresponding to $A$ in $\FreeF$. The \emph{corank} of $A$ is the rank of any complementary free factor of $A$.

In this section, we will define a coarsely Lipschitz map $\Psi_{w,b,B} \colon \FreeF^{\uparrow}(A) \to \Z$ and use it to prove our main theorem in the next section.  

\begin{lemma}\label{connected}
 If corank of $A$ is at least 3, then $\FreeF^{\uparrow}(A)$ is connected. 
\end{lemma}
\begin{proof}

Let $G = T/\free$ for a tree $T \in \relCV{A}$ with vertex $v \in G$ stabilized by $A$. Let $P$ and $Q$ be two connected subgraphs of $G$ that contain $v$ and are not trees. Then the fundamental group of $P$ and $Q$ as graph of groups determine two free factors $\dot{P}$ and $\dot{Q}$ in $\FreeF^{\uparrow}(A)$. We claim that $d_{\FreeF^{\uparrow}(A)}(\dot{P}, \dot{Q}) \leq 6$ (see \cite[Section 3]{BF:Hyperbolicity}). Indeed, let $P'$ be a connected subgraph of $G$ containing $P$ and all edges of $G$ except one, say $p$. Similarly let $Q'$ be a connected subgraph of $G$ contanining $Q$ and containing all edges of $G$ except one, say $q$. Then $P' \cap Q' = G \setminus \{p,q\}$. If $P' = Q'$, then $\dot{P}, \dot{P'}=\dot{Q'}, \dot{Q}$ is a path of length 2 in $\FreeF^{\uparrow}(A)$. Suppose $P' \neq Q'$ and $P' \cap Q'$ is connected. Since corank of $A$ is at least 3, $P' \cap Q'$ is not a tree and contains a loop $R$ based at $v$ with $\dot{R}\in \FreeF^{\uparrow}(A)$. Then $\dot{P}, \dot{P'}, \dot{R}, \dot{Q'}, \dot{Q}$ is a path of length 4 in $\FreeF^{\uparrow}(A)$. Suppose $P' \cap Q'$ is not connected. Then the complement of the interior of the edges $p$ and $q$ has exactly two components because $P', Q'$ are connected subgraphs containing all but one edge each. If the component containing $v$ contains a loop $R$ based at $v$, then we get a path $\dot{P}, \dot{P'}, \dot{R}, \dot{Q'}, \dot{Q}$ of length 4 in $\FreeF^{\uparrow}(A)$. Otherwise, the other component contains a loop $R$ containing an end point of both $p$ and $q$. See \Cref{Distance6}. Then the subgraphs $R_p = \{p\} \cup R, R_q = \{q\} \cup R , R_{pq} = \{p, q\} \cup R$ all contain $v$. We get a path $\dot{P}, \dot{P'}, \dot{R_q}, \dot{R_{pq}}, \dot{R_p}, \dot{Q'}, \dot{Q}$ of length 6 in $\FreeF^{\uparrow}(A)$ as desired. Let $\Pi(G)$ denote the collection of free factors in $\FreeF^{\uparrow}(A)$ arising from subgraphs of $G$. Then we just showed that $\Pi(G)$ is a connected set of diameter at most 6.
\begin{figure}[ht]
    \centering{
    \def\svgwidth{.5\columnwidth}
    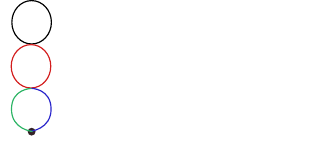
    \caption{Graph $G$ and some subgraphs as in the last case of proof of \Cref{connected}.}
    \label{Distance6}
    }
\end{figure}

Now for $X, Y \in \FreeF^{\uparrow}(A)$, let $G_X$ and $G_Y$ be two marked metric graphs given as the quotient of two $\free$-trees in $\relCV{A}$ such that $X \in \Pi(G_X)$, $Y \in \Pi(G_Y)$. By 
\cite[Proposition 7.5, Theorem 7.7]{FM:FreeProductsTT}, there is a folding path between $G_X$ and $G_Y$ obtained by a continuous parametrization of Stallings foldings. The folding path maps to a path connecting $X$ and $Y$ in $\FreeF^{\uparrow}(A)$. Indeed, if one performs an elementary Stallings fold on $G_X$, that is fold two edges to obtain a new marked graph $G'$, then $\Pi(G_X)$ and $\Pi(G')$ intersect non-trivially. Thus we can produce a path joining $X$ and $Y$ in $\FreeF^{\uparrow}(A)$.  
\end{proof}

Let $X$ be the conjugacy class of a free factor in $\FreeF^{\uparrow}(A)$ and let $T \in \relCV{X}$. The minimal subtree of $w$ in $T$ is the axis of $w$ and the core of $T/\la w \ra$ is a circle, denoted $\core{}$.  The restriction of the metric on $T$ to the axis of $w$ defines a metric on $\core{}$. 

\begin{lemma}\label{lem:w and b} Let $w$ be filling in $B$ and let $T \in \relCV{X}$. Then $w$ is a hyperbolic isometry of $T$. 
\end{lemma}
\begin{proof}
Let $T_B$ be the minimal subtree of $T$ invariant under $B$. Since $B$ is not contained in $X$, $T_B$ is a non-trivial free splitting of $B$. If $w$ is elliptic in $T$, then $w$ is also elliptic in $T_B$. This means that $w$ is contained in a vertex group of $T_B$ which is a free factor of $B$. This is a contradiction since $w$ is filling in $B$. 
\end{proof}

\begin{lemma} \label{lem:cross at least twice}
Let $w$ be filling in $B$ and let $T \in \relCV{X}$. Then $\core{}$ crosses the orbit of every edge of $T$ at least twice. \end{lemma}
\begin{proof}
Suppose $\core{}$ does not cross the orbit of an edge $E$ of $T$. Let $T'$ be the tree obtained by collapsing all edges except the orbit of $E$. Then $T'$ is a free simplicial tree with the vertex group system $Y$ of $T'$ contained in $\FreeF^{\uparrow}(A)$ and $w$ elliptic in $T'$. This is a contradiction to \Cref{lem:w and b} for $T' \in \relCV{Y}$. 

Now suppose $\core{}$ crosses the orbit of some edge $E$ of $T$ exactly once. Then we claim that $w$ is a primitive element of $\free$, which is a contradiction. 
Let $G = T/ \free$ be the quotient graph of groups, containing the edge $E$. Let $\hat{G}$ be a marked $\free$-graph obtained by blowing up the non-trivial vertex stabilizer of $G$. Finally let $\overline{G}$ be a marked rose obtained from $\hat{G}$ by collapsing a maximal tree. Since a loop in $G$ corresponding to $w$ crosses $E$ exactly once, $E$ is not a separating edge of $G$. Therefore, a maximal tree of $\hat{G}$ can be chosen to avoid $E$. Now if $R$ is a marked rose for $\free$ with edges labelled $e_1, \ldots e_n$ and $\alpha$ is a loop that crosses $e_n$ exactly once, then $\{e_1, \ldots, e_{n-1}, \alpha\}$ induces a basis of $\free$. Applying this to $\overline{G}, w$ and $E$, we get the desired contradiction. 
\end{proof}

For any $T \in \relCV{X}$, let $o$ denote the base point, that is, the vertex fixed by $A$. 
For $g \in \free$, let $\mathrm{a}_{T}(g)$ denote the axis or fixed point set of $g$ acting on $T$ and $\tau_T(g)$ denote the translation length.   Recall $b \in B$ is a fixed primitive element. 
Let $L_{T}(b)$ be the shortest path from $o$ to $\mathrm{a}_{T}(b)$ in $T$. We will call it the \emph{leg of $b$} in $T$. Let
\[\Phi_{w,b,B}(T) := \left \lceil \frac{\diam_{T}(L_{T}(b) \cap \mathrm{a}_{T}(w))}{\tau_{T}(w)} \right \rceil .\] 

Informally, $\Phi_{w,b,B}(T)$ is the number of fundamental domains of $w$ crossed by the leg of $b$ in $T$. See \Cref{axis} for an illustration of the definition. For comparison, Clay--Pettet \cite{ClayPettet} define the \emph{relative twist of $T$ and $b$ relative to $w$} as the supremum of the number of fundamental domains of $w$ in the intersection of the axis of $w$ and axes of all conjugates of $b$ in $T$. In our setting, this number is one since $ b \in B$ is primitive and $w$ is filling in $B$. 

\begin{figure}[ht]
    \centering{
    \def\svgwidth{.9\columnwidth}
    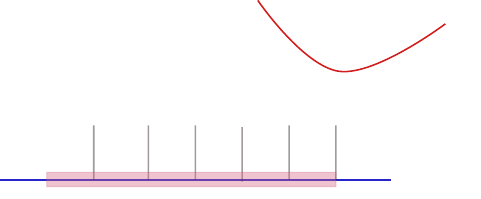
    \caption{In this figure $\Phi_{w,b,B}(T) = 6$.}
    \label{axis}
    }
\end{figure}

\begin{example} 
This example is to illustrate the setup in the above definition. Let $\free = \la a,b,c,d \ra$ and $w=c^2b^2d^2c^2$. The Whitehead graph of $w$ with respect to the basis $\{b,c,d\}$ has no cut vertex therefore it is filling in $\la b,c,d\ra$ (see \cite{LyndonSchupp}). Now consider a new basis $\{a,x,c,d\}$ where $b = d^{-1}c^{-1}xcd$. Then in the new basis $w = c^2 d^{-1}c^{-1} x cd^3c^2$. Let $T$ be the Bass-Serre tree of a graph of groups with one vertex group stabilized by $\la a \ra$ and three loops labeled $x,c,d$. Then in $T$, the axis of $w$ and $b$ are disjoint. 
\end{example}

\begin{example}\label{example}
Let $\free = \la a, b, c, d \ra$ and $w = b^2c^2d^2 \in \la b, c, d \ra$. The Whitehead graph of $w$ with respect to the basis $\{b,c,d\}$ has no cut vertex therefore $w$ is filling in $\la b,c,d \ra$. Let $A = \la a \ra, B = \la b,c,d \ra$ and let $X = [\la a, b \ra]$ and $Y = [\la a, w^N b w^{-N} \ra]$ be two points in $\FreeF^{\uparrow}(A)$ for some $N >0$. 
Let $T_X$ be the Bass-Serre tree of the graph of groups with one vertex stabilized by $X$ and two loops labeled by $c$ and $d$. Similarly, let $T_Y$ be the Bass-Serre tree of the graph of groups with one vertex stabilized by $Y$ and two loops labeled $w^N c w^{-N}$ and $w^N d w^{-N}$. Then $\Phi_{w,b,B}(T_X) = 0$ and $\Phi_{w,b,B}(T_Y) = N$. 
\end{example}

The next proposition shows that $\Phi_{w,b,B}(T)$ coarsely depends only on $X$ and not the choice of $T \in \relCV{X}$.  
\begin{proposition}\label{main prop}
 There exists a constant $C = 2 \mathrm{corank}(A) -2$, 
 such that the following holds: For $X \in \FreeF^{\uparrow}(A)$, let $T$ and $T'$ be two free splittings in $\relFreeS{X}$ with vertex group systems equal to $X$. Then $|\Phi_{w,b,B}(T) - \Phi_{w,b,B}(T')| \leq C$.  
\end{proposition}
\begin{proof}
First suppose that $T$ and $T'$ are distance one apart in $\relFreeS{X}$, that is, there is a collapse map from $T$ to $T'$. Since $w$ is hyperbolic in both $T$ and $T'$, the collapse map changes the number of fundamental domains of $w$ crossed by the leg of $b$ by at most one. Therefore, $|\Phi_{w,b,B}(T) - \Phi_{w,b,B}(T')| \leq 1$.  

Now consider $T'$ as a point in $\relCV{X}$, with some metric and, up to changing $T$ to another tree in a simplex in $\relCV{X}$ containing $T$, consider an optimal morphism $f \colon T \to T'$, where $T$ has the pull back metric. Let $(T_t)_{t \in [0,L]}$ be a folding path with natural parametrization guided by $f$ where $T_0 = T$ and $T_{L} = T'$. 
Note that $\core{}$ is a loop. Let 
$\mathrm{a}_T(w)$ denote the axis of 
$w$ in $T$ and let $\tau_T(w)$ be the translation length of $w$ in $T$. 

By \Cref{lem:cross at least twice}, $\core{}$ crosses the orbit of every edge of $T$ at least twice. Therefore, volume of $\core{}$, which is same as $\tau_T(w)$, is at least twice the covolume of $T$. Let $E$ be any edge of $T$ not in $\mathrm{a}_T(w)$. Then under the folding map $f$, $E$ cannot fold over $\mathrm{a}_T(w)$ for a length more than $\tau_T(w)$ because the length of $E$ is strictly less than $\tau_T(w)$.  

We claim that $|\Phi_{w,b,B}(T) - \Phi_{w,b,B}(T')|$ is uniformly bounded. For $T$, let $G$ denote the quotient graph under the action of $\free$. The fundamental group of $G$, as a graph of groups, is a free product $X * F_m$, where $m$ is the corank of $X$. Choose a maximal tree in $G$. Then the edges not in this maximal tree correspond to a basis for $F_m$. Fix one such partial basis coming from $G$ and one from $G' = T'/\free$. 

Since an edge of $T$ folds over at most one fundamental domain of $w$, we get that the image under $f^*$ of a basis element of $\pi_1(G)$, as chosen above, acquires a suffix or prefix $w^{C}$ for $C$ at most one plus the number of edges in a maximal tree in $G$. The latter is at most $2 \mathrm{corank}(A)-3$. 
This implies that the difference in the number of fundamental domains of $w$ crossed by the leg of $b$, in $T$ and $T'$, is uniformly bounded by $C = 2 \mathrm{corank}(A) -2$.  
\end{proof}

We are now ready to define a coarsely Lipschitz map from $\FreeF^{\uparrow}(A) \to \Z$. A map $\phi \colon X \to Y$ between two metric spaces is called \emph{coarsely Lipschitz}, if there are constants $C_1 >0, C_2 \geq 0$ such that $d_Y(\phi(x), \phi(x')) \leq C_1 d_X(x,x') + C_2$. Choose a complementary free factor $B$ (of $A$) of rank at least 3, and, choose $b, w \in B$ such that $b$ is primitive in $B$ and $w$ is a cyclically reduced filling element in $B$ that is not a power of another element. Define $\Psi_{w,b,B} \colon \FreeF^{\uparrow}(A) \to \Z$ by setting $\Psi_{w,b,B}(X)$ equal to $\Phi_{w,b,B}(T)$ for any $T \in \relFreeS{X}$ with vertex group system equal to $X \in \FreeF^{\uparrow}(A)$.

\begin{proposition}\label{coarsely Lip}
 For a free factor $A$ of $\free$ of corank at least 3, the map $\Psi_{w,b,B} \colon \FreeF^{\uparrow}(A) \to \Z$ is coarsely Lipschitz. 
\end{proposition}
\begin{proof}
Consider two vertices $X$ and $Y$ in $\FreeF^{\uparrow}(A)$ joined by an edge, with $Y \subset X$, up to conjugation. Pick $T_X \in \relFreeS{X}$ with vertex group system $X$ and $T_Y \in \relFreeS{Y}$ with vertex group system $Y$, such that there is a proper forest of $T_Y$ which is collapsed to obtain $T_X$. This forest is the $\free$-orbit of a minimal subtree of $X$ in $T_Y$. Since $w$ is hyperbolic in both $T_X$ and $T_Y$, under the collapse map a fundamental domain of the axis of $w$ in $T_X$ maps to a unique non-degenerate fundamental domain of the axis of $w$ in $T_Y$. Thus $|\Phi_{w,b,B}(T_X) - \Phi_{w,b,B}(T_Y)| \leq 1$. 
Now let $X$ and $Y$ be two vertices of $\FreeF^{\uparrow}(A)$ joined by a geodesic $X=X_0, X_1, \ldots, X_k = Y$ and corresponding trees $T_X, T_{X_1}, \ldots, T_{X_{k-1}}, T_Y$. We will drop the subscripts $w,b,B$ for the remainder of the proof. We have \[|\Phi(T_X) - \Phi(T_Y)| \leq |\Phi(T_X) -\Phi(T_{X_1})| + \ldots + |\Phi(T_{X_{k-1}}) -\Phi(T_Y)| \leq d_{\FreeF^{\uparrow}(A)}(X,Y)\]

By \Cref{main prop}, $|\Psi (X) - \Psi(Y)| \leq |\Phi(T_X) - \Phi(T_Y)| + 2C$. Thus we have \[|\Psi (X) - \Psi(Y)| \leq d_{\FreeF^{\uparrow}(A)}(X,Y) + 2C\]
 and hence $\Psi_{w,b,B}$ is a coarsely Lipschitz map. 
\end{proof}

\section{Proof of main theorem}
We are now ready to prove the main theorem.

\begin{theorem}
Let $\FreeF$ be the free factor complex of free group of rank $n \geq 4$. There exists a constant $C = 2n-4$ and a family of loops $c_N$ of combinatorial length 4 in $\FreeF^{(1)}$ such that the following holds: whenever $P$ is a triangulation of a disc and $f\colon P \to \FreeF^{(2)}$ is a simplicial map with $f|_{\partial P}$ mapping bijectively onto $c_N$, then $P$ must have at least $(N-2C)/3$ triangles. 
\end{theorem}

\begin{proof} Let $\free = \la a, b, a_3, \ldots, a_n \ra$.  
Let $w \in B = \la b, a_3, \ldots, a_n \ra$ be cyclically reduced, filling in $B$ and not a power of another element. 
For any $N>0$, let $c_N$ be the length 4 loop in $\FreeF_n$ with vertices $A_0 = [\la a,b \ra], A_1 = [\la a \ra], A_N = [\la a, w^Nbw^{-N}\ra], A_2 = [\la b \ra]$ (see \Cref{loop}).  Let $P$ be a triangulation of a disc $D^2$ and $c \colon P \to \FreeF_n$ a simplicial map such that $c|_{\partial D^2} = c_N$. 

Since $A_1$ is a rank one free factor, $\FreeF^{\uparrow}(A_1)$ is the full link of $A_1$ in $\FreeF$. By \Cref{coarsely Lip}, the map $\Psi_{w,b,B} \colon \FreeF^{\uparrow}(A_1) \to \Z$ is coarsely Lipschitz. Combined with the calculation in \Cref{example}, the distance between $A_0$ and $A_N$ in the link of $A_1$ is at least $N - 2C$ where $C = 2(n -1)-2 = 2n -4$.  

Let $x_i$ be the pre-image of $A_i$ on the boundary of the disc $D$. Then in $P$, there is an edge path from $x_0$ to $x_N$ of length at least $N - 2C$. If we count every triangle thrice, then we count every edge at least once. Therefore, there are at least $(N -2C)/3$ many triangles in $P$. Thus, for arbitrary $N>0$, the length 4 loop $c_N$ requires at least $(N -2C)/3$ triangles to be capped off. 
\end{proof}

\section{The case \texorpdfstring{$n=3$}{rank 3}} \label{n=3}
The free factor complex for the free group of rank 3 as defined in the introduction is 1-dimensional, that is, a graph. There are many loops in $\FreeF_3$ that are not contractible and hence `capping them off' as in \Cref{subsec:CII} does not make sense. We change the complex, \emph{up to quasi-isometry}, by adding some edges and 2-simplices to obtain a new complex, which we show does not satisfy a combinatorial isoperimetric inequality. 

Let $\FreeF_3'$ be obtained from $\FreeF_3$ by adding edges between two vertices $A = [\la a, x\ra]$ and $B = [\la a, y\ra]$ whenever $a, x, y$ is a basis of $\free_3$. Note, such an $A$ and $B$ are distance 2 apart in $\FreeF_3$ and hence adding these edges does not change the quasi-isometry type of the complex. Basically, we have triangulated the loops of length 6 shown in \Cref{F3}. Now add a 2-simplex whenever we see the 1-skeleton of a 2-simplex. Then $\FreeF_3'$ is quasi-isometric to $\FreeF_3$. Consider the loop $c_N$ in $\FreeF_3'$ as in \Cref{loop}. 

\begin{figure}[ht]
    \centering{
    \def\svgwidth{.5\columnwidth}
    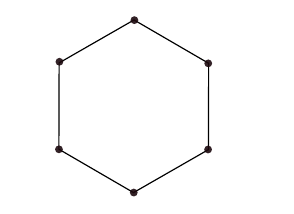
    \caption{Adding edges and 2-simplices to $\FreeF_3$.}
    \label{F3}
    }
\end{figure}

Let $\free_3 = \la a, b, c \ra$, $A = [\la a \ra]$ and $w$ filling, cyclically reduced element in $\la b, c \ra$ that is not a power of another element. The complex $\FreeF^{\uparrow}(A)$ is not connected as a subset of $\FreeF_3$ since all the free factors containing $A$ are rank 2. An argument similar to \Cref{connected} using folding paths, can be used to show that $\FreeF^{\uparrow}(A)$ is in fact a connected subset of $\FreeF_3'$. 

For $X \in \FreeF^{\uparrow}(A)$, there is only one free splitting $T_X$ with vertex group system equal to $X$, up to equivariant isomorphism. Therefore, we set $\Psi_{w,b,B}(X)$ equal to $\Phi_{w,b,B}(T_X)$. Now we show that $\Psi_{w,b,B} \colon \FreeF^{\uparrow}(A) \to \Z$ is a Lipschitz map. Let $X$ and $Y$ be distance one in $\FreeF^{\uparrow}(A)$, that is we may assume $X = [\la a, x \ra], Y = [\la a, y \ra]$ and $\{a, x, y\}$ is a basis of $\free_3$. Let $T_X$ be the Bass-Serre tree of the graph of groups with vertex group $X$ and edge labeled $y$, and vice versa for $T_Y$. Let $T$ be a common refinement of $T_X$ and $T_Y$ described as the Bass-Serre tree of the graph of groups with one vertex stabilized by $[\la a \ra]$ and two loops labeled $x$ and $y$. Then there are collapse maps $p_X\colon T \to T_X$ and  $p_Y\colon T \to T_Y$. Defining $\Phi_{w,b,B}(T)$ as before, we see that $|\Phi_{w,b,B}(T) - \Phi_{w,b,B}(T_X)| \leq 1$ and same for $Y$ since $w$ is hyperbolic in all three trees. Therefore, $|\Phi_{w,b,B}(T_Y) - \Phi_{w,b,B}(T_X)| \leq 2$ which implies that $\Psi_{w,b,B}$ is Lipschitz. Combined with the calculation similar to \Cref{example}, the distance between $A_0 = [\la a,b \ra]$ and $A_N = [\la a, w^N bw^{-N} \ra]$ in the link of $A_1 = [\la a \ra]$ is at least $N/2$. Therefore, the same argument as before shows that at least $N/6$-triangles are needed to cap off $c_N$.

\bibliographystyle{alpha}
\bibliography{ref}

\end{document}